\documentclass[12pt]{article}%{amsart}
\usepackage{amstext, amsmath, amsthm, amssymb}
\usepackage{amsmath}
\usepackage{geometry} % see geometry.pdf on how to lay out the page. There's lots.
\usepackage{color}

\numberwithin{equation}{section}
\usepackage[nottoc]{tocbibind}
\usepackage{hyperref}
\usepackage{makeidx}
\makeindex

\newtheorem{thm}{Theorem}[section]
\newtheorem{cor}[thm]{Corollary}
\newtheorem{lem}[thm]{Lemma}
\newtheorem{prop}[thm]{Proposition}
\newtheorem{defin}[thm]{Definition}
\newtheorem{rem}[thm]{Remark}
\newtheorem{example}[thm]{Example}
\newtheorem{ass}[thm]{Assumption}

\usepackage{colortbl}

\title{Convex topological algebras via linear vector fields and Cuntz algebras}

\author{Wolfgang Bock, Vyacheslav Futorny, Mikhail Neklyudov}

\AtEndDocument{\bigskip{\footnotesize%
  (W.\,Bock) \textsc{Technomathematics Group,
University of Kaiserslautern,
P. O. Box 3049, 67653 Kaiserslautern, Germany} \par
  \textit{E-mail address}: \texttt{bock@mathematik.uni-kl.de} \par
  \addvspace{\medskipamount}
 (V.\,Futorny) \textsc{Instituto de Matematica e Estatistica, Universidade de Sa\~{o} Paulo, Caixa Postal 66281,\\ Sa\~{o} Paulo, CEP 05315-970, Brasil} \par
  \textit{E-mail address}: \texttt{vfutorny@gmail.com}  \par
   \addvspace{\medskipamount}
  (M.\,Neklyudov) \textsc{Instituto de Ci\^{e}ncias Exatas, Departamento de Matematica, UFAM, Manaus, CEP 69077-000, Brasil} \par
  \textit{E-mail address}: \texttt{misha.neklyudov@gmail.com}
}}

\date{}

\begin{document}
\maketitle

\begin{abstract}
Realization by linear vector fields is constructed for any Lie algebra which admits a biorthogonal system and for its any suitable representation. The embedding into  Lie algebras of linear vector fields is analogous to the classical Jordan-Schwinger map. A number of examples of such Lie algebras of linear vector fields is computed. In particular, we obtain examples of the twisted Heisenberg-Virasoro Lie algebra and the Schr\"odinger-Virasoro Lie algebras among others. 
More generally, we construct an embedding  of an arbitrary locally convex topological algebra into the Cuntz algebra.

\medskip

\medskip

\medskip

\end{abstract}

\section{Introduction}
Interest to Lie algebras of vector fields goes back to Sophus Lie in his study of differential operators.  The importance of  Lie algebras of  vector fields in geometry comes from 
the classical result of Shanks and
Pursell \cite{ShanksPursell1954} that the   smooth structure on a manifold is determined by the Lie algebras of  smooth vector fields on it. Lie algebras of algebraic vector fields were studied extensively throughout the years, see \cite{Grabowski1978}, \cite{Siebert1996},  \cite{Jordan2000}, \cite{BilligNilsson2019}, \cite{BilligFutorny2016}, \cite{BilligFutorny2018} and references therein.
These are the Lie algebras of vector fields
which are modules over the corresponding rings of functions. 
Well known four Cartan type Lie algebras are important examples of $\mathbb Z$-graded infinite dimensional Lie algebras of finite growth. The Cartan type Lie algebra $W_n$ of  vector fields on $n$-dimensional torus can be constructed as the derivation algebra of the polynomial algebra. Other Cartan type Lie algebras are 
  subalgebras of $W_n$ which preserve certain differential forms.

In this paper we focus on a class of Lie algebras of linear vector fields. They are realized by differential operators of degree at most $1$. For any Lie algebra which admits a biorthogonal system (e.g.  any separable locally convex Hausdorff Lie algebra) we construct an embedding to 
 Lie algebras of linear vector fields, the same can be done for any
 suitable representation (Theorem \ref{thm-diff-op} and Theorem \ref{thm:LieAlgRep}). The constructed embedding into the Lie algebra of linear vector fields resembles  the classical Jordan-Schwinger map. In Definition \ref{def:DAgeral} we give an alternative map which does not require  the existence of a  biorthogonal system.
 
We provide various examples of  Lie algebras of linear vector fields which arise via such construction. In particular, a Lie algebra of linear vector fields can be associated with any 
Riemannian manifold $M$ and   a Hilbert space of square integrable (with respect to the standard volume measure on $M$) vector fields.  The case $M=\mathbb{R}$ leads to the well known class of twisted Heisenberg-Virasoro algebras.  We use the construction to obtain a class of representations by linear vector fields  for
the Schr\"odinger-Virasoro Lie algebras. 

Finally, we generalize our construction 
for arbitrary locally convex topological algebra and their  homotopes and obtain an embedding into the Cuntz algebra (Corollaries \ref{cor-cuntz_1} and \ref{cor-cuntz_2}). We also give an explicit representation by linear vector fields for any convex topological finite dimensional algebra using the representation of Cuntz algebra constructed by Dutkay \cite{Dutkay2014}.

\section*{Acknowledgments}

V.\,F.\ is supported in part by the CNPq (304467/2017-0) and by the Fapesp (2018/23690-6). M.\,N.\ is supported in part by the CAPES (Brasil)-- Finance code 001.

\section{Preliminaries}

All vector spaces are considered over the field $\mathbb{R}$ of real numbers.

\begin{ass}\label{ass:Space_0}
Assume that $V$ and $W$ are topological vector spaces in duality  with pairing $<\cdot,\cdot>_{V,W}$, $V$ is separable and there exists a biorthogonal system $\{e_k,f_k\}_{k=1}^{\infty}$, where $\{e_k\}_{k=1}^{\infty}\subset V$, $\{f_k\}_{k=1}^{\infty}\subset W$, that is $V=\overline{span\{e_k,k\in\mathbb{N}\}}$ and $<e_k,f_j>=\delta_{kj},k,j\in\mathbb{N}$.  
\end{ass}
Assumption \ref{ass:Space_0} implies that for any $x\in V$ we have the following presentation
\begin{equation}
x=\sum\limits_{i\in \mathbb{N}}<x,f_i> e_i.\label{eqn:BasisExpansion}
\end{equation}

\begin{example}\label{ex:LocConvTopAlgebra}
Let $V$ be a separable locally convex Hausdorff topological vector space, $W=V^{\prime}$ its topological dual and $<\cdot,\cdot>_{V,V^{\prime}}$ duality between $V$ and $V^{\prime}$. The  existence of a biorthogonal system in this case has been shown, for instance, in \cite{Klee1958}.
\end{example}

\begin{example}
Let $V$ is a separable topological vector space endowed  with bilinear continuous symmetric non degenerate form $Q:V\times V\to \mathbb{R}$.
Then there exists a countable set $A$ and a sequence $\{e_k\}_{k\in A}$ of normalized orthogonal elements with respect to the form $Q$, that is
\[
Q(e_i,e_j)=\delta_{ij}\sigma_i,i,j\in A=A^{+}\cup A^{-},
\]
where
\[
\sigma^i=\left\{
\begin{array}{cc}
1, & i\in A^+\\
-1,& i\in A^-.
\end{array}
\right.
\]
Indeed, it is enough to apply the Gram-Schmidt orthogonalization procedure to a countable dense set of $V$. In this case, $W=V$ and $f_k=\sigma^k e_k, k\in\mathbb{N}$.
\end{example}

 Let $\mathbb{R}[\overline{x}]$ be the space of real polynomials in infinity many variables $x_1, x_2, \ldots $.  
Denote by $\mathcal A$ the Weyl algebra with generators $x_1, x_2, \ldots $ and $\partial_1, \partial_2, \ldots $ subject the relations 
$$\partial_i x_j- x_j \partial_i=\delta_{ij},$$ 
and its completion $\widehat{\mathcal A}$ with infinite linear combinations of differential operators on $\mathbb{R}[\overline{x}]$.  We will identify 
$\partial_i$ with the differential operator $\frac{\partial}{\partial x_{i}}$ for all $i$.
Finally, let 
$\widehat{\mathcal A}_l$ be the subspace of $\widehat{\mathcal A}$ consisting of \emph{linear} differential operators, that is
operators of form
 $$\sum_{i,j}a_{ij}x_i^{t_i}\partial_j^{r_j}, $$ with $1\leq t_i+r_j\leq 2$ for all $i,j$. 

If $V$ and $W$ are vector spaces then 
denote by 
$\mathcal{L}(V,W)$ the space of linear maps from $V$ to $W$.

\section{Linear differential operators}

Now we define our key operators

\

\begin{defin}\label{def:DA}
Let $D=D_V:\mathcal{L}(V,W)\to \widehat{\mathcal A}_l $, $\partial:V\to \widehat{\mathcal A}_l $, $\bar{\partial}:W\to \widehat{\mathcal A}_l $ be mappings defined as follows:
\begin{equation}\label{DA}
D(A):= \sum\limits_{\alpha,\beta\in \mathbb{N}}<Ae_{\alpha},f_{\beta}> x_{\alpha}\frac{\partial}{\partial x_{\beta}}, A\in \mathcal{L}(V,V),
\end{equation}
\begin{equation}
\partial(h):= \sum\limits_{\alpha\in \mathbb{N}}<h,f_{\alpha}> \frac{\partial}{\partial x_{\alpha}}, h\in V.
\end{equation}
\end{defin}
\begin{equation}
\bar{\partial}(r):= \sum\limits_{\alpha\in \mathbb{N}}<e_{\alpha},r>  x_{\alpha}, r\in W.
\end{equation}

\

\begin{rem}
%In the papers and books of physicists, see for instance 
  The Jordan--Schwinger map for a matrix $X=(X_{ij})_{i,j=1}^n$  is defined as follows \cite[pp. 212--213]{BidernharnLouck1981}:
  
\begin{equation}\label{def:DX}
X\mapsto \sum\limits_{i,j=1}^n X_{ij} a_i a_j^{*},
\end{equation}

where $\{a_i, a_j^*\}_{i,j=1}^n$ are boson creation and annihilation operators, i.e. the elements of the canonical commutation relations (CCR) algebra.  Hence \ref{def:DA} can be described using Jordan--Schwinger maps. For that, 
 instead of general elements of the CCR algebra, consider their realizations by operators of multiplication and derivation to underline the linearity of operators. Furthermore, if we set $V=W=\mathbb{R}^n$ in the definition of \ref{def:DA} then  \eqref{def:DX}  in our notation is given by
\[
X\mapsto D(X^*).
\]
Consequently, as the following Lemma \ref{lem:CommRelations} shows,  \ref{def:DA} defines an anti-homomorphism, while  \eqref{def:DX} gives a homomorphism.
\end{rem}

\

We have
\begin{lem}\label{lem:CommRelations}
\begin{eqnarray}
&[D(A),D(B)] = D([B,A]),\label{eqn:CommRelations_1}\\
&[\partial(f),D(A)] = \partial(Af),\label{eqn:CommRelations_2}\\
&[\partial(f),\partial(g)] = 0\label{eqn:CommRelations_3}\\
&[D(A),\bar{\partial}(r)] = \bar{\partial}(A^*r)\label{eqn:CommRelations_4}\\
&f,g\in V, r\in W, A,B\in \mathcal{L}(V,W).\nonumber
\end{eqnarray}
\end{lem}
\begin{proof}
Indeed, 
\begin{eqnarray*}
&[D(A),D(B)] = \sum\limits_{i_1,i_2,j_1,j_2\in  \mathbb{N}} <A e_{i_1}, f_{j_1}>
<B e_{i_2}, f_{j_2}>[x_{i_1}\frac{\partial}{\partial x_{j_1}},x_{i_2}\frac{\partial}{\partial x_{j_2}}]\\
&= \sum\limits_{i_1,i_2,j_1,j_2\in  \mathbb{N}} <A e_{i_1}, f_{j_1}>
<B e_{i_2}, f_{j_2}> (x_{i_1}\delta_{i_2,j_1}\frac{\partial}{\partial x_{j_2}}-x_{i_2}\delta_{i_1,j_2}\frac{\partial}{\partial x_{j_1}})\\
&= \sum\limits_{i_1,j_2\in  \mathbb{N}} (\sum\limits_{j_1\in  \mathbb{N}} <A e_{i_1}, f_{j_1}>
<B e_{j_1},f_{j_2}>)x_{i_1}\frac{\partial}{\partial x_{j_2}}\\
&-\sum\limits_{i_2,j_1\in  \mathbb{N}}(\sum\limits_{j_2\in  \mathbb{N}} <A e_{j_2}, f_{j_1}>
<B e_{i_2},f_{j_2}>)x_{i_2}\frac{\partial}{\partial x_{j_1}}\\
&=  \sum\limits_{i_1,j_2\in  \mathbb{N}} <(BA-AB)e_{i_1}, f_{j_2}>x_{i_1}\frac{\partial}{\partial x_{j_2}}= D([B,A])
\end{eqnarray*}
where we have used representation \eqref{eqn:BasisExpansion} and sequential continuity of the pairing. The second and the fourth commutation relations are proved similarly. The third one is obvious.
\end{proof}

\

\begin{rem}\label{rem:Mpointmotion}
Relation \eqref{eqn:CommRelations_2} can be rewritten in the form $ad(-D(A))\partial(f)=\partial(Af)$. Consequently, from the operator identity
\[
[A_1\ldots A_n, B]=[A_1,B]A_2\ldots A_n+\ldots+A_1A_2\ldots A_{n-1}[A_n,B],
\]
follows that
\[
ad(-D(A))\left(\partial(f_1)\ldots\partial(f_n)\right)=\sum\limits_{k=1}^n \partial(f_1)\ldots\partial(A f_k)\ldots \partial(f_n),f_1,\ldots,f_n\in V.
\]
Thus the operator $ad(-D(A))$ acts on the linear space generated by $\partial(f_1)\ldots\partial(f_n)$, $ f_1,\ldots,f_n\in V$, as an operator of $n$-times motion
$A^{\otimes^n}=\sum\limits_{k=1}^n Id\otimes\ldots\otimes A\otimes\ldots\otimes Id$.
\end{rem}

\

\begin{cor}\label{cor:InvariantSubspace}
\begin{itemize}
\item[(i)]
Let $A\in \mathcal{L}(V,V)$ and $W\subset V$ is an invariant subspace of $A$ i.e. $A\in \mathcal{L}(W,W)$. Define 
\[
X:=\{\psi\in End(\mathbb{R}[\overline{x}])| \partial(f)\psi=0,\quad \forall f\in W\}.
\]
Then $X$ is an invariant subspace of $D(A)$.
\item[(ii)] Let $X$ be an invariant subspace for $D(A)$. Then
\[
W:=\{f\in V| \partial(f)\psi=0,\forall \psi \in X\}
\]
is an invariant subspace of $A$.
\end{itemize}
\end{cor}
\begin{cor}\label{cor:Semigroup}
\[
\partial(e^{tA}f)=e^{-tD(A)}\partial(f)e^{t D(A)}, f\in V
\]
\end{cor}
\begin{proof}
Immediately follows from formula \eqref{eqn:CommRelations_2}.
\end{proof}

Assume now that $V=W$. Then $\mathcal{L}(V,V)$ has a natural structure of a Lie algebra. 
Let $\mathfrak g$ be a Lie subalgebra of $\mathcal{L}(V,V)$. Then the restriction of $-D$  onto $\mathfrak g$ defines a representation of $\mathfrak g$ by linear differential operators.  Hence, we have
\

\begin{thm}\label{thm-diff-op}
 Let  $\mathfrak g$ be an arbitrary Lie algebra and $\rho: \mathfrak g \rightarrow End(V)$  a faithful representation of $\mathfrak g$. Then  $-D\circ \rho$ and $D^*\circ \rho$ give  embeddings of $\mathfrak g$ into $\widehat{\mathcal A}_l$, and hence, define representations of 
$\mathfrak g$ by linear vector fields.
\end{thm}

\

%\begin{rem}
%Instead of the set $End(\mathbb{R}[\overline{x}])$ of endomorphisms we could have considered CCR-algebra defining mapping $D$ and $\partial$ using creation and annihilation operators instead of multiplication and derivative correspondently.   
%\end{rem}
%
%\
%
%{\color{red} More details
%\begin{rem}
%Instead of the set $End(\mathbb{R}[\overline{x}])$ of endomorphisms we could have considered CCR-algebra defining mapping $D$ and $\partial$ using creation and annihilation operators instead of multiplication and derivative correspondently.   
%\end{rem}
%}

Let $\mathfrak g$ be a Lie algebra with the center $Cent(\mathfrak g)$, which satisfies the assumption 
\ref{ass:Space_0}. Set $V=\mathfrak g$ and 
restrict the mapping $D=D_V$ on the subspace 
\[
Z:=\{ad(v):= [v,\cdot], v\in V \}\subset \mathcal{L}(V,V)
\] 
of operators of adjoint representation.  We obtain
 
\begin{equation}\label{eqn:Embedding}
\tilde{D}:= D\circ ad: \mathfrak g\to End(\mathbb{R}[\overline{x}]), 
\tilde{D}(v)=\sum\limits_{\alpha,\beta\in A}<[v,e_{\alpha}],f_{\beta}> x_{\alpha}\frac{\partial}{\partial x_{\beta}}, v\in \mathfrak g.
\end{equation}

Consequently, Lemma \ref{lem:CommRelations} implies

\begin{thm}\label{thm:LieAlgRep}
For any Lie algebra $\mathfrak g$ satisfying the assumption \eqref{ass:Space_0} (with $V=\mathfrak g$) there exists an embedding, given by formula \eqref{eqn:Embedding}, of ${\mathfrak g}/Cent({\mathfrak g})$  into the semidirect product $\tilde{D}({\mathfrak g})\ltimes \partial({\mathfrak g})\subset End(\mathbb{R}[\overline{x}]) $ of linear differential operators.
\end{thm}

\begin{proof}
We have
\begin{eqnarray}
&[\tilde{D}(u),\tilde{D}(v)] = \tilde{D}([v,u]),\label{eqn:CommRelationsLie_1}\
&[\partial(h),\tilde{D}(v)] = \partial([v,h]),\label{eqn:CommRelationsLie_2}\\
&[\partial(h),\partial(g)] = 0\label{eqn:CommRelationsLie_3}\\
&h,g,u,v\in \mathfrak g.\nonumber
\end{eqnarray}
If $\tilde{D}(v)=0$ then equality \eqref{eqn:CommRelationsLie_2} implies that 
$\partial([v,h])=0, h\in\mathfrak g$ and, consequently, $[v,h]=0, h\in\mathfrak g$ i.e.  $v\in Cent(\mathfrak g)$.
\end{proof}
\begin{rem}
Example \eqref{ex:LocConvTopAlgebra} shows that condition \eqref{ass:Space_0} is satisfied for any separable locally convex Hausdorff Lie algebra.
\end{rem}
Corollary \eqref{cor:InvariantSubspace} in this case shows that for each $v\in \mathfrak g$ there is one to one correspondence between subspaces of the Lie algebra invariant under action of $ad(v)$ and subspaces of $End(\mathbb{R}[\overline{x}])$ invariant with respect to $\tilde{D}(v)$.

\

%{\color{red} Not clear what to do with this
%\section{Analytical index}
%Let $i_a(A), A\in\mathcal{L}(V,V)$ analytical index of $A$ i.e. 
%\[
%i_a(A)=dim\,ker (A)-dim\,ker (A^*).
%\]
%Then
%\[
%i_a(A)=dim\,ker\ad[D(A)]-dim\,ker\ad[D(A^*)]
%\]
%Indeed, it immediately follows from the formula \eqref{eqn:CommRelations_2}.
%Under some conditions we could expect that $D(A^*)=D^*(A)$. The question would be: Can we get that $\ad[D(A^*)]=(\ad[D(A)])^*$?
%For this we need to introduce additional scalar product on the set of operators invariant with respect to $\ad$.
%}

\section{Mapping $D$ as an extension of the algebraic adjoint  operator}

Notice that $D(A)|_{V^*}=A^*$ and, consequently, our construction is an extension of the algebraic adjoint from the class of linear continuous functionals to a more general class of functions. Then we can define $D(A)$ in the following fashion:
\begin{defin}\label{def:DAgeral}
\begin{itemize}
\item[(i)]  $D(A)|_{V^*}=A^*$.

\item[(ii)]$\forall k\in\mathbb{N},\, l_1,\ldots, l_k\in V^*,\,\phi \in C^{\infty}(\mathbb{R}^k)$ define
\[
D(A)\left[\phi(l_1,\ldots,l_k)\right]:=\sum\limits_{m=1}^k \frac{\partial\phi}{\partial x_m}(l_1,\ldots,l_k)A^*l_m
\]

\item[(iii)]  Let $\mu$ be a Radon Gaussian measure on the space $V$ and $H(\mu)\subset V$ be the Cameron-Martin space of $\mu$. For general $f\in L^2(V, d\mu)$ we define 
\[
D(A)f(x):=\lim\limits_{t\to 0}\frac{f(x+tAx)-f(x)}{t}, x\in V,
\]
whenever limit exists. Note that whenever $f$ is a cylindrical function this definition coincides with (ii). 
By the density of the set of cylindrical functionals in $L^2(V,d\mu)$  and Cauchy-Schwartz inequality we can extend $D(A)$ to the Sobolev space $W^{4,1}(\mu)$ (see, for instance, \cite[chapter $5$, p.211]{Bogachev98}),  for all $A$ such that 
\[
\int\limits_V |Ax|_{H(\mu)}^4d\mu<\infty.
\]

\end{itemize}
\end{defin}

Therefore, our construction provides an extension of any representation to a linear vector field representation. Note that Theorem \ref{thm:LieAlgRep} represents an extension of the adjoint representation, first to coadjoint, and then to a vector field representation. The advantage of the definition \ref{def:DAgeral} comparing to the definition \ref{def:DA} is that we do not  require the existence of a biorthogonal system. We  only require the existence of a non empty space of linear continuous functionals. 

%I am thinking that may be it can be used in some way, for the method of %orbits of Kirillov? In the method of Kirillov we define symplectic form %on the orbit of linear continuous functional. Now we have the way to %extend this form to be defined for the orbit of more general function.

\

\section{Analog of the Killing form}
\begin{defin}
Define a bilinear form $\varepsilon:\mathcal{L}(V,V)\times\mathcal{L}(V,V)\to End(\mathbb{R}[\overline{x}])$ as follows
\[
\varepsilon(A,B):= D(A)D(B)+D(AB), A,B\in \mathcal{L}(V,V).
\]
\end{defin}
\begin{rem}
The form above measures how far the map $D$ is from being an (anti)homomorphism of algebras.
% i.e. if we we change the sign in the definition of $D$ then natural definition would be $\varepsilon(A,B):= D(A)D(B)-D(AB)$.
\end{rem}
\begin{lem}
We have
\begin{itemize}
\item[(i)] The form $\varepsilon$ is symmetric.
%\[
%\varepsilon(A,B)=\varepsilon(B,A), A,B\in \mathcal{L}(V,V).
%\]
\item[(ii)]
\[
[D(A),\varepsilon(B,C)]=-\varepsilon([A,B],C)-\varepsilon(B,[A,C]), A,B,C\in \mathcal{L}(V,V).
\]
\end{itemize}
\end{lem}
\begin{proof}
Statement (i) follows immediately  from  \eqref{eqn:CommRelations_1}, while .
(ii)
 follows from  \eqref{eqn:CommRelations_1}  and operator identity $[A,BC]=[A,B]C+B[A,C]$.

\end{proof}

For a Lie algebra $\mathfrak g$ define the  bilinear map 
$\tilde{\varepsilon}:\mathfrak g\times \mathfrak g\to End(\mathbb{R}[\overline{x}])$ as follows:
\[
\tilde{\varepsilon}(u,v):=\varepsilon(ad(u),ad(v)),\, u,v\in \mathfrak g.
\]

Then we have
\begin{cor}\label{cor:PropertiesTE}
\begin{itemize}
\item[(i)]
$\tilde{\varepsilon}(u,v)=\varepsilon(v,u), u,v\in \mathfrak g.$
\item[(ii)]
$
[D(ad(u)),\tilde{\varepsilon}(v,w)]=-\tilde{\varepsilon}(ad(u)v,w)-\tilde{\varepsilon}(v,ad(u)w), u,v,w\in \mathfrak g.
$
\end{itemize}
\end{cor}
This bilinear map $\tilde{\varepsilon}$ takes values in  second order operators (infinite dimensional) or (if we use creation and annihilation operators on Fock space to define $D$) in the set of linear transformations on Fock space (bosonic or fermionic).

Let $\tau: End(\mathbb{R}[\overline{x}])\rightarrow \mathbb{R}$ be a \emph{weight} function, which is a linear function whose restriction on the image of $\tilde{\varepsilon}$ is symmetric, that is $\tau(AB)=\tau(BA)$ for all $A,B\in End(\mathbb{R}[\overline{x}])$). Then one can use 
 the bilinear map $\tilde{\varepsilon}$ and $\tau$ to  define the following analog of the Killing form on $\mathfrak g$:
\begin{defin}\label{def:BilinearformRV}
\[
B:\mathfrak g\times \mathfrak g\to\mathbb{R}, B(u,v)=(\tau \circ \tilde{\varepsilon})(u,v), u,v\in \mathfrak g.
\]
\end{defin}

Therefore we can conclude from corollary \ref{cor:PropertiesTE} that
\begin{cor}\label{cor:PropertiesTE_2}
\begin{itemize}
\item[(i)]
The form $B$ is symmetric: \[
B(u,v)=B(v,u), u,v\in \mathfrak g.
\]
\item[(ii)]
\[
B(ad(u)v,w)+B(v,ad(u)w)=0, u,v,w\in \mathfrak g.
\]
\end{itemize}
\end{cor}
%i.e. we can realize $ad$ as antisymmetric transformation of bilinear symmetric form on our Lie algebra. 
\

If the form $B$ is non-degenerate then one  constructs a 2-cocycle and consequently a central extension of  $\mathfrak g$ in the usual way.

\

\begin{cor}\label{cor:PropertiesTE_3}
Let $u\in \mathfrak g, u\neq 0$ and $\phi_u: \mathfrak g\times \mathfrak g\to\mathbb{R}$ defined by
$\phi_u(w,z)=B(ad(u)w,z)$ for $w,z\in \mathfrak g$.
Then
\begin{itemize}
\item[(i)]
\[
\phi_u(w,z)=-\phi_u(z,w), w,z\in \mathfrak g.
\]
\item[(ii)]
\[
\phi_u(y,[w,z])+\phi_u(w,[z,y])+\phi_u(z,[y,w])=0, y,w,z\in \mathfrak g,
\]
\end{itemize}
i.e. $\phi_u$ is a 2-cocycle.
\end{cor}

\

%\begin{proof}
%Immediately follows from symmetry of $B$ and antisymmetry of $ad(u)$ w.r.t. $B$.
%\end{proof}
%\section{Central extension of the form $B$}
%
%Let $\widetilde{\Lambda}=\Lambda\oplus \mathbb{C}e$, $e\in Cent(\Lambda)$--central extension of Lie algebra $\Lambda$ with Lie bracket given by
%\[
%\widetilde{[u,v]}:=[u,v]+c(u,v)e,
%\]
%where $c$ satisfies
%\begin{itemize}
%\item[(i)]
%\[
%c(u,v)=-c(v,u)
%\]
%\item[(ii)]
%\[
%c(u,[v,w])+c(v,[u,w])+c(w,[u,v])=0
%\]
%\end{itemize}
%\begin{lem}
%The bilinear symmetric form $\widetilde{B}$ corresponding to the Lie algebra $\widetilde{\Lambda}$ has the following formula
%\begin{eqnarray}
%\widetilde{B}(u,v) &= B(u,v)+Tr\left[c(u,\cdot)(\partial(e)+D(e))D(ad(v))\right]\nonumber\\
%&+
%Tr\left[c(v,\cdot)(\partial(e)+D(e))D(ad(u))\right]\nonumber\\
%&+Tr\left[c(u,\cdot)c(v,\cdot)(D(e^2)+\partial(e)^2)\right],\, u,v\in\widetilde{\Lambda}
%\end{eqnarray}
%where $B$--bilinear symmetric form for the Lie algebra $\Lambda$.
%\end{lem}
%\begin{proof}
%For each $u\in\widetilde{\Lamda}$ we can define 
%\end{proof}

\section{Examples}
In this section we consider examples of computation of vector fields $D(A)$. 

\subsection{Operators on $L^2(0,2\pi)$}
\begin{trivlist}
\item[(1)]
Let $H=\{\phi\in L^2(0,2\pi)\| \phi(0)=\phi(2\pi)=0\}$ be a Hilbert space with scalar product $(f,g)_H=\frac{1}{\pi}\int\limits_{0}^{2\pi}f(x)g(x)\,dx,f,g\in H$ and orthogonal basis $\{e_n=\sin (n\cdot)\}_{n=1}^{\infty}$;
$A: \mathcal{D}(A)\subset H\mapsto H$, $\mathcal{D}(A)=\{\phi\in W^{2,2}(0,2\pi)\cap H  | \ \phi''(0)=\phi''(2\pi)=0\}$, $Af=\lambda\frac{\partial^2}{\partial_x^2}+(1-\lambda)v\frac{\partial}{\partial_x},v\in H$.

Since $v\in H$ we have that $v=\sum\limits_{m=1}^{\infty}c_m e_m$ with
$|v|_H=\sum\limits_{m=1}^{\infty}c_m^2<\infty$. Now we can deduce that
\[
A e_m=-\lambda m^2e_m+(1-\lambda)\sum\limits_{n=1}^{\infty}c_n e_n\partial_x e_m
\]
Consequently, we have
\[
(Ae_m,e_k)_H=-\lambda m^2\delta_{mk}+(1-\lambda)\sum\limits_{n=1}^{\infty}c_n (e_n\partial_x e_m,e_k)_H.
\]
We can calculate that 
\begin{eqnarray}
(e_n\partial_x e_m,e_k)_H &= \frac{m}{\pi}\int\limits_0^{2\pi}\sin nx\cos mx\sin kx,dx=\frac{m}{2}(\delta_{n,k-m}+\delta_{n,k+m}-\delta_{n,-k-m}-\delta_{n,m-k})\nonumber\\
&=\frac{m}{2}(\delta_{n,k-m}+\delta_{n,k+m}-\delta_{n,m-k}), n,m,k\in\mathbb{N}.
\end{eqnarray}
Hence,
\[
(Ae_m,e_k)_H=-\lambda m^2\delta_{mk}+\frac{m(1-\lambda)}{2}(c_{m+k}+c_{k-m}-c_{m-k}),m,k\in\mathbb{N}
\]
where we use notation $c_k:=0,k\leq 0$.
Now we can conclude that 
\begin{eqnarray}
D(A) &= \sum\limits_{m=1}^{\infty}\left(\frac{(1-\lambda)}{2}mc_{2m}-\lambda m^2\right)y_m\partial_{y_m}\nonumber\\
&+\frac{(1-\lambda)}{2}\sum\limits_{n<m,m,n\in\mathbb{N}}c_{m+n}(my_m\partial_{y_n}+n y_n \partial_{y_m})\nonumber\\
&+c_{m-n}(n y_n \partial_{y_m}-m y_m \partial_{y_n})
\end{eqnarray}
\item[(2)]

Let $H=\{\phi\in L^2(0,2\pi)\| \phi(0)=\phi(2\pi)=0\}$ be a Hilbert space with scalar product $(f,g)_H=\frac{1}{\pi}\int\limits_{0}^{2\pi}f(x)g(x)\,dx,f,g\in H$ and orthogonal basis $\{e_n=\sin (n\cdot)\}_{n=1}^{\infty}$;
$A:\mathcal{D}(A)\subset H\mapsto H$, $A=x^2\frac{\partial}{\partial_x}$.
The flow corresponding to vector field $A$ will be $X_t(x)=\frac{x}{1-tx},  t< \frac{1}{2\pi}$.

We can calculate that 
\[
\bar{A}_{nm}:=(A e_n,e_m)_H=
\left\{
\begin{array}{ccc}
\frac{4\pi nm}{n^2-m^2} & ,& n\neq m\\
-\pi & , & n=m
\end{array}
\right.
n,m\in\mathbb{N}
\]
Then we can conclude from the corollary \ref{cor:Semigroup} that
\[
\frac{1}{\pi}\int\limits_{0}^{2\pi} e_n(X_t(x))e_m(x)\, dx= (e^{t\bar{A}})_{nm},n,m\in\mathbb{N}.
\]
Notice that the matrix $\bar{A}=-\pi I+B$, where $B$ is antisymmetric matrix. Consequently $e^{t\bar{A}}$ is defined for all $t\geq 0$ and exponentially fast convergent to $0$ as $t\rightarrow\infty$.

\end{trivlist}

\subsection{Operators on $L^2(M,TM,d\lambda)$}

Let $M$ be a Riemannian manifold, $H=L^2(M,TM,d\lambda)$ be a Hilbert space of square integrable (with respect to the standard volume measure on $M$) vector fields,
$\nabla_{\cdot}(\cdot):H\times H\to H $--covariant derivative. Define $\widetilde{D}(X):=-D(\nabla_X), X\in H$. Then
\begin{eqnarray}
&[\widetilde{D}(X),\partial(Y)] = \partial(\nabla_{X} Y),\label{eqn:rel_1}\\
&{[\widetilde{D}(X),\partial(fY)]} = [\widetilde{D}(fX),\partial(Y)]+\partial(Ydf(X)),\label{eqn:rel_2} \\
&\widetilde{D}(R(X,Y)) = [\widetilde{D}(X),\widetilde{D}(Y)]-\widetilde{D}([X,Y]),X,Y\in H,f\in C^{\infty}(M),\label{eqn:rel_3}
\end{eqnarray}
where $R$ is Riemannian curvature tensor. 
If in addition torsion tensor is zero we have that 
\begin{equation}
[\widetilde{D}(X),\partial(Y)] = [\widetilde{D}(Y),\partial(X)] +\partial([X,Y])\label{eqn:rel_4}
\end{equation}

Thus operators $\{\partial(Y),\widetilde{D}(X)\},X,Y\in H$ generate certain Lie algebra which depends upon manifold $M$. 

\

In the particular case, when our manifold $M=\mathbb{S}^1$ is a circle we can identify tangent fields $X$ with scalar functions $x:\mathbb{S}^1\to \mathbb{R}$ as follows
\[
x=x(\theta)\longleftrightarrow X=x(\theta)(-\sin\theta,\cos\theta),\theta\in\mathbb{S}^1.
\]
Consequently, $\nabla_X\longleftrightarrow x(\cdot)\frac{d}{d\theta}$, $R=T=0$ and relations \eqref{eqn:rel_1} and \eqref{eqn:rel_3} become
\begin{eqnarray}
&[\widetilde{D}(x),\partial(y)] = \partial(x \frac{dy}{d\theta}),\label{eqn:rel_1a}\\
%&{[\widetilde{D}(x),\partial(fy)]} = [\widetilde{D}(fx),\partial(y)]+\partial(yx\frac{df}{d\theta}),\label{eqn:rel_2a} \\
&[\widetilde{D}(x),\widetilde{D}(y)]=\widetilde{D}(x \frac{dy}{d\theta}-y \frac{dx}{d\theta}),\label{eqn:rel_3a} \\
&[\partial(x),\partial(y)]=0,x,y,f\in C^{1}(\mathbb{S}^1).
%&[\widetilde{D}(x),\partial(y)] = [\widetilde{D}(y),\partial(x)] +\partial(x \frac{dy}{d\theta}-y \frac{dx}{d\theta}),x,y,f\in C^{1}(\mathbb{S}^1),\label{eqn:rel_4a} 
\end{eqnarray}
Relations  \eqref{eqn:rel_2} and  \eqref{eqn:rel_4} are easily deduced from \eqref{eqn:rel_1} and \eqref{eqn:rel_3} in this case.

\subsection{Heisenberg-Virasoro Lie algebra}
 Let us put in the framework of the previous example $M=\mathbb{R}$ and consider the family $X_n=x^n\partial_x, n\in\mathbb{Z}$ of vector fields on $\mathbb{R}$. Let us denote $\widetilde{D}_n=\widetilde{D}(X_n), \partial_m=\partial(X_m),n,m\in \mathbb{Z}$. Clearly in this case $R=T=0$ and we have
\begin{eqnarray}
&[\widetilde{D}_n,\partial_{m}] = m\partial_{n+m-1},\label{eqn:rel_1b}\\
&[\widetilde{D}_n,\widetilde{D}_m]=(m-n)\widetilde{D}_{n+m-1},\label{eqn:rel_3b} \\
&[\partial_{n},\partial_{m}]=0,n,m\in \mathbb{Z}. 
\end{eqnarray}

Here

\[
\tilde{D}_n=-\sum\limits_{i,j=1}^{\infty}(x^n\frac{de_i}{dx},e_j)_{L^2(\mathbb{R})}y_i\partial_{y_j}
\]
\[
\partial_m=\sum\limits_{i=1}^{\infty}(x^m,e_j) _{L^2(\mathbb{R})}\partial_{y_j}
\]
where $\{e_i\}_{i=1}^{\infty}$ orthonormal basis for $L^2(\mathbb{R}) $ (and we identify functions and vector fields on $\mathbb{R}$)

Consider the  shift of indexes by setting $d_n=\widetilde{D}_{n+1}$. Then  
\begin{eqnarray}
&[d_n,\partial_{m}] = m\partial_{n+m},\\
&[d_n, d_m]=(m-n)d_{n+m}, n,m\in \mathbb{Z}. 
\end{eqnarray}

The Lie algebra spanned by $\{d_n, \partial_{m}, n,m\in \mathbb{Z}\}$ satisfying relations above is the centerless \emph{twisted Heisenberg-Virasoro algebra} which was extensively studied (e.g. \cite{ShenSu2007},  \cite{Billig2003}, \cite{LvZhao2006}, \cite{LiuJing2008}).
 Of course, a realization of 
twisted Heisenberg-Virasoro algebra by linear vector fields is well known. 

\

\subsection{Schr\"odinger-Virasoro  Lie algebra}
Now we consider Schr\"odinger-Virasoro  Lie algebras which play important role in statistical physics \cite{He}. 
These algebras can be realized as the semidirect product of the centerless Virasoro algebra and a certain module 
of the intermediate series.   

Let $s=0,\frac{1}{2}$ and $\rho\in\mathbb{Q}$. Denote by 
$\mathcal{L}[s,\rho]$ the complex Lie algebra with  basis $\{L_n, Y_p, n\in \mathbb{Z},p\in\mathbb{Z}+s\}$ satisfying the following relations
\begin{eqnarray}
&[L_m,L_n] = (n-m)L_{n+m}\label{rel_1}\\
&[L_m,Y_p] = (p-m\rho)Y_{m+p}\label{rel_2}\\
&[Y_p,Y_q] = 0\,\, ,m,n\in \mathbb{Z},p,q\in\mathbb{Z}+s.\label{rel_3}
\end{eqnarray}

Note that $\mathcal{L}[0,0]$ is the centerless twisted Heisenberg-Virasoro Lie algebra. 
The Schr\"odinger-Virasoro algebras and their representations were studied in many papers (e.g. \cite{RU2006}, \cite{Liu2016}, \cite{LiSu2008}).

Let $H$ be a separable infinite dimensional Hilbert space and $\mathcal{H}:=A(D,H)$ be the Banach space of analytical functions on the disk $D=\{z\in\mathbb{C}||z|\leq 1 \}$ endowed with the  uniform topology. Define the
system of operators $A_m:\mathcal{H}\mapsto \mathcal{H}$ as follows

\[
A_m f:=e^{-imz}(-m\rho f+i\frac{df}{dz}),z\in D
\]
Then easy calculation shows that
\[
[A_m,A_n]=(n-m)A_{n+m}\,\, ,m,n\in \mathbb{Z},
\]
that is the system $\{A_m,m\in\mathbb{Z}\}$ gives a representation of  the first Witt algebra (the centerless Virasoro algebra) given by relations \eqref{rel_1}.  Let $g\in H$ and $\{e_p:= g e^{-ipz}\}_{p\in\mathbb{Z}+s}\subset \mathcal{H}$. Then we can define $\partial$ and $\bar{D}:=-D$ as before. Then $\bar{D}$ is a homomorphism and we have
\[
[\bar{D}(A_m),\bar{D}(A_n)]=(n-m)\bar{D}(A_{n+m})\,\, ,m,n\in \mathbb{Z}.
\]
Furthermore, clearly we have
\[
[\partial(e_p), \partial(e_q)] = 0\, ,p,q\in\mathbb{Z}+s.
\]
Moreover, 
\[
[\bar{D}(A_m),\partial(e_p)] = \partial(A_m e_p)=(p-m\rho) \partial(e_{m+p}),p\in  \mathbb{Z}+s,m\in \mathbb{Z}.
\]
Thus the system $\{\partial(e_p),\bar{D}(A_m),p\in  \mathbb{Z}+s,m\in \mathbb{Z}\}$ defines a representation of the Lie algebra $\mathcal{L}[s,\rho]$ by linear vector fields.

%{\color{red} Can we say anything about the irreducibility? Of course, the whole space of analytic functions is too big. But perhaps some natural subspace will be irreducible?}

\

\subsection{ Dynamical systems} Let $X$ be a countable set, for instance integer lattice $\mathbb{Z}^d$, $h:X\to X$, and $\{x_n\}_{n=0}^{\infty}$ a sequence defined by
\[ 
\begin{array}{ccc}
x_{n+1} &=& h(x_{n}), n\geq 0\nonumber\\
x_0 &=& x\in X.
\end{array}
\]
For any $\phi:X\to \mathbb{R}$ define
\[
A\phi(x):=\phi(h(x)),S_n\phi(x):=\phi(x_n),\, x\in X,n\geq 0.
\]
Hence, $S_n=A^n,n\geq 0$.
We set $V=W=\mathbb{R}^{X}$. Using the standard orthonormal basis $\{e_k\}_{k\in X}$ in $V$  we obtain
\[
D(A)=\sum\limits_{l\in X} x_l\frac{\partial}{\partial_{x_{h_l}}},
\]
\[
\partial(f)=\sum\limits_{l\in X} f(l)\frac{\partial}{\partial_{x_l}}.
\]
Consequently, we can conclude that
\[
\partial(S_n\phi)=(-ad(D(A)))^n\partial(\phi)=(-D(A))^n\partial(\phi)D(A)^n,n\geq 0.
\]
\begin{rem}

Similarly to the remark \ref{rem:Mpointmotion} we can deduce that the evolution $S_n^{\otimes^m}:=-ad(D(A)))^n\left(\partial(\phi_1)\dots\partial(\phi_m)\right)$ describes the $m$-point evolution. 
\end{rem}

\

\section{Generalization to the homotopes of topological algebras}

\begin{ass}\label{ass:Space_01}
Assume that $V$ and $W$ are topological vector spaces in duality  with pairing $<\cdot,\cdot>_{V,W}$ and $L\in \mathcal{L}(V,V)$-- continuous linear operator. Let $\{e_{\alpha}\}_{\alpha\in I}$ be a basis of $V$. Relative to this basis we have 
\[
Lx=\sum\limits_{\alpha\in I} L^{\alpha}(x)e_{\alpha},
\]
where we assume the set $I$  either countable or $I=[0,1]$. In the latter case, the sum will be understood as an integral  with respect to $\alpha$. Furthermore, we assume that $L^{\alpha}\in W$, $\alpha\in I$.
\end{ass}

\

Denote by $\mathcal{O}_{\infty}$ the Cuntz algebra  (\cite{Cuntz1977})
with generators $\{s_i\}_{i\in\mathbb{N}}$ subject to the following relation:
$s_j^*s_j=1$, $s_j^*s_k=0$, $j\neq k$, $j,k\in\mathbb{N}$.

\

\begin{defin}
Let $D:\mathcal{L}(V,V)\to \mathcal{O}_{\infty}$, $\partial:V\to \mathcal{O}_{\infty}$, $\bar{\partial}:W\to \mathcal{O}_{\infty}$ be mappings defined as follows:
\begin{equation}
D(A):=\sum\limits_{\alpha,\beta\in I}L^{\beta}(Ae_{\alpha}) s_{\alpha}s_{\beta}^*, A\in \mathcal{L}(V,V),
\end{equation}
\begin{equation}
\partial(h):=\sum\limits_{\alpha\in I}L^{\alpha}(h) s_{\alpha}^*, h\in V,
\end{equation}
\begin{equation}
\bar{\partial}(f):=\sum\limits_{\alpha\in I}<e_{\alpha},f> s_{\alpha}, f\in W.
\end{equation}
\end{defin}

\

Consequently, we have:

\

\begin{lem}\label{lem:CommRelationsAlg}
Under assumption \ref{ass:Space_01} we have
\begin{eqnarray}
&D(A)D(B) = D(BLA),\label{eqn:CommRelations_1Alg}\\
&\partial(h)D(A) = \partial(ALh),\label{eqn:CommRelations_2Alg}\\
&D(A)\bar{\partial}(f) = \bar{\partial}(A^*L^*f),\label{eqn:CommRelations_3Alg}\\
&\partial(h)\bar{\partial}(g) = <Lh,g>\label{eqn:CommRelations_4Alg}\\
&h\in V,f,g\in W, A,B\in \mathcal{L}(V,V),\nonumber
\end{eqnarray}
where adjoint is taken with respect to the duality $<\cdot,\cdot>$.
\end{lem}
\begin{proof}
\begin{eqnarray*}
D(A)D(B) &=\sum\limits_{i,j,k,l}L^j(Ae_i)L^l(Be_k)s_is_j^*s_ks_l^*\\
&=\sum\limits_{i,l}\left(\sum\limits_{j}L^j(Ae_i)L^l(Be_j)\right)s_i s_l^*\\
&=\sum\limits_{i,l}L^l\left(B\left[\sum\limits_{j}L^j(Ae_i)e_j\right]\right)s_i s_l^*\\
&=\sum\limits_{i,l}L^l(BLAe_i)s_i s_l^*=D(BLA),
\end{eqnarray*}
\begin{equation*}
\partial(h)D(A)=\sum\limits_{i,j,k}L^k(h)L^j(Ae_i)s_k^*s_is_j^*=\sum\limits_{j,k}L^k(h)L^j(Ae_k)s_j^*=\partial(ALh),
\end{equation*}
\begin{equation*}
D(A)\bar{\partial}(f)=\sum\limits_{i}<LAe_i,f>s_i=\bar{\partial}(A^*L^* f), A,B\in \mathcal{L}(V,V),,h\in V, f\in W.
\end{equation*}
Similarly,
\begin{equation*}
\partial(h)\bar{\partial}(f)=\sum\limits_{\alpha\in I}L^{\alpha}(h)<e_{\alpha},f>=<Lh,f>, h\in V, f\in W.
\end{equation*}
\end{proof}

\

\begin{example}\label{example_0}
Assume that $V$ is separable with  system $ \{e_k\}_{k=1}^{\infty}\subset V$  such that
 $V=\overline{sp\{e_k,k\in\mathbb{N}\}}$. Let $\{f_k\}_{k=1}^{\infty}\subset W$.
Define
\begin{equation*}
L(x)=\sum\limits_{i\in \mathbb{N}}<x,f_i> e_i, x\in V.
\end{equation*}
In particular case when the system $\{e_k,f_k\}_{k=1}^{\infty}$ is biorthogonal, that is $<e_k,f_j>=\delta_{kj},k,j\in\mathbb{N}$, the operator 
$L$ becomes the identity operator, and definitions of $\partial$ and $D$ are reduced to the following:

\begin{equation}
D(A):=\sum\limits_{\alpha,\beta\in \mathbb{N}}<Ae_{\alpha},f_{\beta}> s_{\alpha} s_{\beta}^*, A\in \mathcal{L}(V,V),
\end{equation}
\begin{equation}
\partial(h):=\sum\limits_{\alpha\in \mathbb{N}}<h,f_{\alpha}> s_{\alpha}^*, h\in V.
\end{equation}
\end{example}
From now on we assume that $V=X$ is a topological algebra in duality with $X^*$. Denote by $l_a:X\to X$, $l_a(x)=ax $ the operator of multiplication by $a\in X$. Consider the operators
$A,B,L\in \mathcal{L}(X,X)$ such that $Lx=l_{\rho} x$, $Ax=l_a x$ and $B x=l_b x$,  where $ a,b,\rho\in X$ for any $x\in X$. Then we have from Lemma \ref{lem:CommRelationsAlg}: 

\

\begin{cor}\label{cor-cuntz_1}
Assume that $X$ is a topological separable locally convex Hausdorff algebra. Then we have

\begin{eqnarray}
&D(l_a)D(l_b) = D(l_{b\rho a}),\label{eqn:CommRelations_1a}\\
&\partial(h)D(l_a) = \partial(a\rho h),\label{eqn:CommRelations_2a}\\
&D(l_a)\bar{\partial}(f) = \bar{\partial}(l_a^*l_{\rho}^*f),\label{eqn:CommRelations_3_a}\\
&\partial(h)\bar{\partial}(g) = <g,\rho h>\label{eqn:CommRelations_4_a}\\
&a,b,\rho,h\in X, f,g\in X^*,\nonumber
\end{eqnarray}
\end{cor}

\begin{proof}
There exists  a biorthogonal system
$\{e_j\}_{j\in\mathbb{N}}\subset X,\{f_j\}_{j\in\mathbb{N}}\subset X^*$ \cite{Klee1958}. Consequently, we have an expansion \eqref{eqn:BasisExpansion}
and
\begin{equation}
Lx=\rho x=\sum\limits_{i\in \mathbb{N}}<l_{\rho}x,f_i> e_i, x\in V.\label{eqn:LBasisExpansion}
\end{equation}
Hence, the assumption \eqref{ass:Space_01} is satisfied. Now the result is a direct consequence of Lemma \eqref{lem:CommRelations}.
\end{proof}

\

Let us denote by $X^{\rho}$ the $\rho$-homotope of a topological separable locally convex Hausdorff algebra $X$, that is $X$ with a modified  product $(a,b)\mapsto a\rho b$.

The relation \eqref{eqn:CommRelations_1a} means that we have an (anti)homomorphism of $X^{\rho}$ into the Cuntz algebra $\mathcal{O}_{\infty}$. Relations \eqref{eqn:CommRelations_2a}  and \eqref{eqn:CommRelations_3_a} give us a representation of $X^{\rho}$ and a representation of the adjoint of $X^{\rho}$, respectively. 

\

\begin{cor}\label{cor-cuntz_2}
If $X^{\rho}$ contains the identity then the 
correspondence $X^{\rho}\ni a\mapsto D(l_a)\in \mathcal{O}_{\infty}$ is injective.
\end{cor}

\begin{proof}
By the linearity it is sufficient to show that  $D(l_a)=0$ implies $a=0$.
If $D(l_a)=0$ then it follows from the  identity \eqref{eqn:CommRelations_2a}  that $\partial(a\rho h)=0,\,\,\forall h\in X$. Furthermore, applying the identity \eqref{eqn:CommRelations_4_a} we get
\[
<g,\rho a\rho h>=0,\,\,\forall h\in X,g\in X^*,
\]
which leads to $\rho a\rho h=0$, for all $h\in X$. If $X^{\rho}$ has the identity $e$ then we put $h=e$ and immediately get $\rho a=0$. Consequently, $a=0$.
\end{proof}

\

Hence, one can construct new representations of  $X^{\rho}$ by restricting the representations of the Cuntz algebra.  
Interesting \emph{permutation} representations of Cuntz algebras  have been studied in \cite{BratelliJoergensen99}.

\begin{rem}
It would be interesting to find a representation of $\mathcal{O}_{\infty}$ as a limit of finite difference operators (which appear in the representations of the $q$-deformation of classical CCR algebra) and, more generally, connect the version of the Jordan-Schwinger map above with the classical Jordan-Schwinger map with $q$-deformation in a way that every term of the family conserves the $q$-commutator.
\end{rem}

Note that the map  $D(l_{\cdot})$ is an antihomomorphism which conserves (up to a sign) any $q$-commutator $(a,b)\mapsto ab-qba$. 

\

Next  we construct the map $D(l_{\cdot})$ for topological separable locally convex Hausdorff algebra $\mathcal{A}$ of dimension $n$  using the representation of the Cuntz algebra $\mathcal{O}_{n}$ constructed in  \cite{Dutkay2014}. First, we need some definitions.

\begin{defin}(\cite{Dutkay2014}) 
Let $Y$ be a  topological compact space, $\mu$-- a Borel probability measure on $Y$, $r:Y\to Y$ -- an $n$-to-$1$ Borel measurable map, i.e. $|r^{-1}(z)|=n$ for $\mu$-almost all $z\in Y$. We assume that $\mu$ is a strongly invariant measure with respect to $r$, that is the condition
\[
\int fd\mu=\frac{1}{n}\int\sum\limits_{r(\omega)=r(z)}f(\omega) d\mu(z),f\in C(Y)
\]
is satisfied.
\end{defin}

\

\begin{defin}(\cite{Dutkay2014}) 
A quadrature mirror filter (QMF) for $r$ is a function $m_0$ in $L^{\infty}(Y,\mu)$ with the property
\begin{equation}\label{QMF}
\frac{1}{n}\sum\limits_{r(\omega)=z} |m_0(\omega)|^2=1,z\in Y.
\end{equation}

A QMF basis is a set of $n$ QMF's $m_0,m_1,\ldots,m_{n-1}$ such that
\begin{equation}\label{QMFOrtho}
\frac{1}{n}\sum\limits_{r(\omega)=z} m_i(\omega)\overline{m_j}(\omega)=1\delta_{ij},\,i,j\in \{0,1,\ldots,n-1\},z\in Y.
\end{equation}
\end{defin}

\
We have
\begin{prop}(\cite{Dutkay2014})
Let $\{m_i\}_{i=0}^{i=n-1}$  be a QMF basis. Define the following operators on $L^2(Y,d\mu)$:
\[
S_i(f) = m_i (f \circ r), i = 0,\ldots,n-1.
\]
Then the operators $\{S_i\}_{i=0}^{n-1}$ are isometries and they form a representation of the Cuntz algebra $\mathcal{O}_{n}$. The adjoint of $S_i$ is given by a formula
\[
S_i^*(f)(z)=\frac{1}{n}\sum\limits_{r(\omega)=z}\overline{m_i}(\omega)f(\omega),i = 0,\ldots,n-1,z\in Y.
\]
\end{prop}

\

Consequently, the formula for $D(l_{\cdot})$ in this wavelet representation looks as follows:

\begin{equation}\label{WaveletRepresentation}
D(l_a)f(z)=\frac{1}{n}\sum\limits_{r(\omega)=r(z)}\left(\sum\limits_{i,j=0}^{n-1}<ae_i,f_j>m_i(z)\overline{m_j}(\omega)\right)f(\omega), z\in Y, a\in\mathcal{A},
\end{equation}
for $f\in L^2(Y,d\mu)$.

The operator $D(l_a)$ in \eqref{WaveletRepresentation} describes some nonlinear dynamical system.

\

\end{document}